\documentclass[12pt,amsymb,fullpage]{amsart}
\usepackage{amssymb,amscd,pstricks}

\newtheorem{theorem}{Theorem}[section]
\newtheorem{defn}[theorem]{Definition}

\newtheorem{lemma}[theorem]{Lemma}

\newtheorem{eple}[theorem]{Example}
\newtheorem{rmk}[theorem]{Remarks}
\newtheorem{dsc}[theorem]{Discussion}
\newtheorem{nota}[theorem]{Notation}

\newsavebox{\indbin}
\savebox{\indbin}{\begin{picture}(0,0)
\newlength{\gnu}
\settowidth{\gnu}{$\smile$} \setlength{\unitlength}{.5\gnu}
\put(-1,-.65){$\smile$} \put(-.25,.1){$|$}
\end{picture}}

\newcommand{\be}{\begin{enumerate}}
\newcommand{\bd}{\begin{defn}}
\newcommand{\bt}{\begin{theorem}}
\newcommand{\bl}{\begin{lemma}}
\newcommand{\ee}{\end{enumerate}}
\newcommand{\ed}{\end{defn}}
\newcommand{\et}{\end{theorem}}
\newcommand{\el}{\end{lemma}}

\begin{document}
\title{An Application of Fourier Analysis to Riemann Sums}
\author{Tristram de Piro}
\address{Mathematics Department, Harrison Building, Streatham Campus, University of Exeter, North Park Road, Exeter, Devon, EX4 4QF, United Kingdom}
\begin{abstract}
We develop a method for calculating Riemann sums using Fourier analysis.
\end{abstract}

\maketitle
\begin{section}{Poisson Summation Formula}

\begin{defn}
\label{transforms}

If $f\in L^{1}(\mathcal{R})$, we define;\\

$(f)^{\wedge}(y)=\int_{-\infty}^{\infty}f(x)e^{-2\pi ixy}dx$\\

$(f)_{-}(y)=f(-y)$\\

$(f)^{\vee}(y)=\int_{-\infty}^{\infty}f(x)e^{2\pi ixy}dx$\\

and, if $g\in L^{1}([0,1])$, $m\in\mathcal{Z}$, we define;\\

$(g)^{\wedge}(m)=\int_{0}^{1}g(x)e^{-2\pi ixm}dx$\\

\end{defn}

\begin{rmk}
\label{inversion}

If $f\in\mathcal{S}(\mathcal{R})$, we have that;\\

$f(x)=\int_{-\infty}^{\infty}(f)^{\wedge}(y)e^{2\pi ixy}dy$, $(x\in\mathcal{R})$\\

and, if $g\in C^{\infty}([0,1])$, (\footnote{\label{smooth} By which we mean that $g|_{(0,1)}\in C^{\infty}(0,1)$, and there exist $\{g_{k}\in C[0,1]: k\in\mathcal{Z}_{\geq 0}\}$, such that $g_{k}|_{(0,1)}=g^{(k)}$, and $g_{k}(0)=g_{k}(1)$.}), the series;\\

$\sum_{m\in\mathcal{Z}}(g)^{\wedge}(m)e^{2\pi ixm}$\\

converges uniformly to $g$ on $[0,1]$. See \cite{SS},\cite{dep1} and \cite{dep2}.\\

Also observe that $(f)^{\vee}=(f_{-})^{\wedge}$ and $(f)^{\wedge}=(f_{-})^{\vee}$.

\end{rmk}

\begin{theorem}
\label{poisson}

Let $f\in\mathcal{S}(\mathcal{R})$, and let;\\

$g(y)=\sum_{m\in\mathcal{Z}}f(y+m)$, $(y\in [0,1])$\\

Then $g\in C^{\infty}([0,1])$ and the series\\

$\sum_{m\in\mathcal{Z}}(f)^{\wedge}(m)e^{2\pi iym}$\\

converges uniformly to $g$ on $[0,1]$.\\

In particular;\\

$\sum_{m\in\mathcal{Z}}f(m)=\sum_{m\in\mathcal{Z}}(f)^{\wedge}(m)$\\

\end{theorem}

\begin{proof}
Observe that, as $f\in\mathcal{S}(\mathcal{R})$, for $y_{0}\in [0,1]$, $r\in\mathcal{Z}_{\geq 0}$, $n\geq 2$;\\

$\sum_{m\in\mathcal{Z}}|{d^{r}f\over dy^{r}}|_{y_{0}+m}|$\\

$\leq \sum_{m\in\mathcal{Z}}{C_{r,n}\over (1+|y_{0}+m|^{n})}$\\

$\leq \sum_{m\in\mathcal{Z}}{C_{r,n}\over (1+|m|^{n})}$\\

$\leq C_{r,n}+2C_{r,n}\sum_{m\geq 1}{1\over m^{n}}$\\

$\leq C_{r,n}+2C_{r,n}(1+[{y^{-n+1}\over -n+1}]^{\infty}_{1}$\\

$=C_{r,n}(1+2(1+{1\over n-1})\leq 5C_{r,n}$ $(*)$\\

where $C_{r,n}=sup_{w\in\mathcal{R}}(|w|^{n}{d^{r}f\over dx^{r}}|_{w})$\\

Suppose, inductively, that ${d^{r}g\over dy^{r}}|_{y_{0}}=\sum_{m\in\mathcal{Z}}{d^{r}f\over dy^{r}}|_{y_{0}+m}$, for $y_{0}\in [0,1]$, (\footnote{Given ${d^{r}g\over dy^{r}}$, we interpret ${d^{r+1}g\over dy^{r+1}}|_{0}=lim_{h\rightarrow 0,+}{1\over h}({d^{r}g\over dy^{r}}|_{h}-{d^{r}g\over dy^{r}}|_{0})$}), then, using $(*)$, we have, for $r\geq 1$, that;\\

${d^{r+1}g\over dy^{r+1}}|={d\over dx}(\sum_{m\in\mathcal{Z}}{d^{r}f_{m}\over dy^{r}})=\sum_{m\in\mathcal{Z}}{d^{r+1}f_{m}\over dy^{r+1}}$\\

where $f_{m}(x)=f(x+m)$, for $m\in\mathcal{Z}$. Moreover, for $r\geq 0$;\\

${d^{r}g\over dy^{r}}|_{0}=\sum_{m\in\mathcal{Z}}{d^{r}f_{m}\over dy^{r}}|_{0}=\sum_{m\in\mathcal{Z}}{d^{r}f_{m}\over dy^{r}}|_{1}={d^{r}g\over dy^{r}}|_{1}$\\

It follows that $g\in C^{\infty}[0,1]$. Moreover, we have that, for $n\in\mathcal{Z}$;\\

$(g)^{\wedge}(n)=\int_{0}^{1}g(y)e^{-2\pi iyn}dx$\\

$=\int_{0}^{1}(\sum_{m\in\mathcal{Z}}f(y+m))e^{-2\pi iyn}dx$\\

$=\int_{0}^{1}(\sum_{m\in\mathcal{Z}}f(y+m))e^{-2\pi i(y+m)n}dx$\\

$=\int_{-\infty}^{\infty}f(y)e^{-2\pi iyn}dx=(f)^{\wedge}(n)$\\

Using Remark \ref{inversion}, the series;\\

$\sum_{m\in\mathcal{Z}}\hat{f}(m)e^{2\pi iym}$\\

converges uniformly to $g$ on $[0,1]$ as required.

\end{proof}

\begin{lemma}
\label{inversionconditions}
If $h\in C^{2}(\mathcal{R})$, and there exists $C\in\mathcal{R}$, with;\\

 $sup_{x\in\mathcal{R}}(|x|^{2}|h(x)|,|x|^{2}|h'(x)|,|x|^{2}|h''(x)|)\leq C$\\

then the Inversion theorem holds for $h$. That is $(h)^{\wedge}\in L^{1}(\mathcal{R})$ and;\\

$h(x)=\int_{\mathcal{R}}(h)^{\wedge}(y)e^{2\pi i xy}dy$ $(x\in\mathcal{R})$\\

\end{lemma}

\begin{proof}
The result follows from inspection of the proof in \cite{dep1}, see Remark 0.4.

\end{proof}

\begin{lemma}
\label{converse}
If $h$ satisfies the conditions of Lemma \ref{inversionconditions}, and $f=(h)^{\vee}$, then $(f)^{\wedge}=h$.
\end{lemma}

\begin{proof}

As $h$ satisfies the conditions of Lemma \ref{inversionconditions}, so does $h_{-}$, and, therefore, the inversion theorem holds for $h_{-}$. Then;\\

$((h_{-})^{\wedge})^{\vee}=(((h_{-})^{\wedge})^{\wedge})_{-}=(h_{-})$\\

therefore;\\

$(((h_{-})^{\wedge})^{\wedge})=h$. As $f=h^{\vee}={(h_{-})}^{\wedge}$, we have that;\\

$(f)^{\wedge}=(h_{-})^{\wedge})^{\wedge}=h$\\

\end{proof}

\begin{lemma}
\label{C2}
Let $f$ be given by Lemma \ref{converse}. Then, if there exists $D\in\mathcal{R}$, with;\\

 $sup_{x\in\mathcal{R}}(|x|^{4}|h(x)|,|x|^{4}|h'(x)|,|x|^{4}|h''(x)|)\leq D$\\

 we have that $f\in C^{2}(\mathcal{R})$, and, moreover, there exists a constant $F\in\mathcal{R}$, such that;\\

$sup_{y\in\mathcal{R}}(|y|^{2}|f(y)|,|y|^{2}|f'(y)|,|y|^{2}|f''(y)|)\leq F$.\\

\end{lemma}

\begin{proof}

Letting $E=||h|_{[-1,1]}||_{C[-1,1]}$, we have that, for $y\in\mathcal{R}$, $|x|\geq 1$;\\

$|h(x)e^{2\pi ixy}|=|h(x)|\leq {D\over |x|^{4}}\leq {D\over |x|^{2}}$\\

$|2\pi i x h(x)e^{2\pi ixy}|=2\pi |x||h(x)|\leq {2\pi D\over |x|^{3}}\leq {2 \pi D\over |x|^{2}}$\\

$|-4\pi^{2}x^{2}h(x)e^{2\pi ixy}|=4\pi^{2}|x|^{2}|h(x)|\leq {4\pi^{2} D\over|x|^{2}}\leq {4\pi^{2}D\over |x|^{2}}$\\

and, for $y\in\mathcal{R}$, $|x|\leq 1$;\\

$|x|^{2}|h(x)e^{2\pi ixy}|\leq |h(x)|\leq E$\\

$|x|^{2}|2\pi i x h(x)e^{2\pi ixy}|\leq 2\pi|h(x)|\leq 2\pi E$\\

$|x|^{2}|-4\pi^{2}x^{2} h(x)e^{2\pi ixy}|\leq 4\pi^{2}|h(x)|\leq 4\pi^{2}E$\\

Hence;\\

$sup_{x\in\mathcal{R}}\{|x|^{2}|h(x)e^{2\pi ixy}|,|x|^{2}|2\pi i x h(x)e^{2\pi ixy}|,|x|^{2}|-4\pi^{2}x^{2}h(x)e^{2\pi ixy}|\}$\\

$\leq 4\pi^{2}max(D,E)$\\

and $\{h(x)e^{2\pi ixy},2\pi i x h(x)e^{2\pi ixy},-4\pi^{2}x^{2}h(x)e^{2\pi ixy}\}\subset C(\mathcal{R})$. It follows that, for $y_{0}\in\mathcal{R}$, we can differentiate under the integral sign, to obtain that $\{f(y_{0}),f'(y_{0}),f''(y_{0})\}$ are all defined. By the DCT, using the fact that $-4\pi^{2}x^{2}h(x)\in L^{1}(\mathcal{R})$, we obtain that $f''\in C(\mathcal{R})$, hence, $f\in C^{2}(\mathcal{R})$. Differentiating by parts, using the fact that;\\

$\{h,h',h'',xh,xh',xh'',x^{2}h,x^{2}h',x^{2}h''\}\subset (L^{1}(\mathcal{R})\cap C_{0}(\mathcal{R}))$\\

by the hypotheses of Lemma \ref{inversionconditions} and this Lemma, we have that;\\

$(h'')^{\vee}=-4\pi y^{2}(h)^{\vee}=-4\pi y^{2}f$\\

$(4\pi ih'+2\pi i xh'')^{\vee}=((2\pi ix h)'')^{\vee}=-4\pi y^{2}(2\pi ix h)^{\vee}=-4\pi y^{2}f'$\\

$(8\pi^{2}h+16\pi^{2}xh'+4\pi^{2}x^{2}h'')^{\vee}=((4\pi^{2}x^{2}h)'')^{\vee}$\\

$=-4\pi y^{2}(4\pi^{2}x^{2}h)^{\vee}=-4\pi y^{2}f''$, $(*)$\\

We have, by $(*)$, for $|y|\geq 1$, that;\\

$|f(y)|\leq {|(h'')^{\vee}(y)|\over 4\pi y^{2}}\leq {||h''||_{L^{1}(\mathcal{R})}\over 4\pi y^{2}}\leq {{2D\over 3}+2E''\over 4\pi y^{2}}$\\

$|f'(y)|\leq {|(4\pi ih'+2\pi i xh'')^{\vee}(y)|\over 4\pi y^{2}}$\\

$\leq {||(4\pi ih'+2\pi i xh'')||_{L^{1}(\mathcal{R})}\over 4\pi y^{2}}$\\

$\leq {2||h'||_{L^{1}(\mathcal{R})}+||xh''||_{L^{1}(\mathcal{R})}\over 2 y^{2}}$\\

$\leq {2({2D\over 3}+2E')+D+2E''\over 2 y^{2}}$\\

$|f''(y)|\leq {|(8\pi^{2}h+16\pi^{2}xh'+4\pi^{2}x^{2}h'')^{\vee}(y)|\over 4\pi y^{2}}$\\

$\leq {||(8\pi^{2}h+16\pi^{2}xh'+4\pi^{2}x^{2}h'')||_{L^{1}(\mathcal{R})}\over 4\pi y^{2}}$\\

$\leq {2\pi||h||_{L^{1}(\mathcal{R})}+4\pi||xh'||_{L^{1}(\mathcal{R})}+\pi||x^{2}h''||_{L^{1}(\mathcal{R})}\over y^{2}}$\\

$\leq {2\pi(2{D\over 3}+2E)+4\pi(D+2E')+\pi(2D+2E'')\over y^{2}}$\\

where $E'=||h'|_{[-1,1]}||_{C[-1,1]}$ and $E''=||h''|_{[-1,1]}||_{C[-1,1]}$\\

For $|y|\leq 1$, we have that;\\

$|f(y)|\leq ||h||_{L^{1}(\mathcal{R})}\leq {2D\over 3}+E$\\

$|f'(y)|\leq ||(2\pi i xh)||_{L^{1}(\mathcal{R})}\leq 2\pi D+2E$\\

$|f''(y)|\leq ||-4\pi^{2}x^{2}h||_{L^{1}(\mathcal{R})}\leq 8\pi^{2}D+2E$\\

Hence, we can take $F=max(8\pi^{2}D+2E,{22\pi D\over 3}+4\pi E+8\pi E'+2\pi E'')$

\end{proof}

\begin{defn}
\label{sum3}

Let $f$ be given by satisfying the conditions of Lemmas \ref{converse} and \ref{C2}, we let;\\

$g(y)=\sum_{m\in\mathcal{Z}}f(y+m)$, $(y\in [0,1])$\\

\end{defn}

\begin{lemma}
\label{smooth2}

Let $g$ be given by Definition \ref{sum3}, then $g\in C^{2}[0,1]$.

\end{lemma}

\begin{proof}
Using Lemma \ref{C2} and Weierstrass' M-test, we have that the series;\\

$\sum_{m\in\mathcal{Z}}f(y+m)$, $\sum_{m\in\mathcal{Z}}f'(y+m)$, $\sum_{m\in\mathcal{Z}}f''(y+m)$\\

are uniformly convergent on $[0,1]$. It follows, that $g\in C^{2}(0,1)$, and clearly;\\

$g'_{+}(0)=\sum_{m\in\mathcal{Z}}f'(m)=\sum_{m\in\mathcal{Z}}f'(m+1)=g'_{-}(1)$\\

hence, $g\in C^{2}[0,1]$.

\end{proof}

\begin{lemma}
\label{uniform2}

Let $f\in L^{1}(\mathcal{R})$, such that;\\

$g(y)=\sum_{m\in\mathcal{Z}}f(y+m)$\\

is defined, for $y\in [0,1]$. Then, if $g\in C^{2}[0,1]$, we have that the series  $\sum_{m\in\mathcal{Z}}(f)^{\wedge}(m)e^{2\pi i ym}$ converges uniformly to $g$ on $[0,1]$. In particular;\\

$\sum_{m\in\mathcal{Z}}f(m)=\sum_{m\in\mathcal{Z}}(f)^{\wedge}(m)$\\

\end{lemma}

\begin{proof}

Following through the calculation in Theorem \ref{poisson}, we have that $g\in L^{1}([0,1])$, and $(g)^{\wedge}(m)=(f)^{\wedge}(m)$, for $m\in\mathcal{Z}$. Using the result of \cite{dep2} or \cite{SS}, we obtain the second part, the final claim is clear.

\end{proof}

\begin{lemma}
\label{sumformula}

Let $f$ be given by satisfying the conditions of Lemmas \ref{converse} and \ref{C2}, with respect to $h$, then;\\

$\sum_{m\in\mathcal{Z}}f(m)=\sum_{m\in\mathcal{Z}}h(m)$\\

\end{lemma}

\begin{proof}
Using Lemmas \ref{converse} and \ref{C2}, we have that $g\in C^{2}[0,1]$, where $g$ is defined by \ref{smooth2}, and $(f)^{\wedge}(m)=h(m)$, for $m\in\mathcal{Z}$.  By Lemmas \ref{smooth2} and \ref{uniform2}, we have that;\\

$\sum_{m\in\mathcal{Z}}f(m)=\sum_{m\in\mathcal{Z}}(f)^{\wedge}(m)$\\

Hence;\\

$\sum_{m\in\mathcal{Z}}f(m)=\sum_{m\in\mathcal{Z}}h(m)$\\

as required.

\end{proof}

\begin{lemma}
\label{cot}
If $s\in\mathcal{Z}_{\geq 2}$, $s$ even, then;\\

$\sum_{n=1}^{\infty}{1\over n^{s}}={(-1)^{{s+2\over 2}}(2\pi)^{s}B_{s}\over 2(s!)}$\\

\end{lemma}

\begin{proof}
The proof of this result can be found in \cite{SS}.
\end{proof}

\begin{defn}
\label{h}
If $s\in\mathcal{C}$, with $Re(s)\geq 4$, and $r\in\mathcal{Z}_{\geq 1}$, we define;\\

$h_{s,r}(x)={1\over x^{s}}$,  $(x\geq r)$\\

$h_{s,r}(x)={(-1)^{s}\over x^{s}}={e^{-i\pi s}\over x^{s}}$, $(x\leq -r)$\\

\end{defn}

\begin{rmk}
\label{symmetry}
$h_{s,r}$ is symmetric, that is $h_{s,r}(x)=h_{s,r}(-x)$, for $|x|\geq r$.

\end{rmk}

\begin{lemma}
\label{polynomial}

There exists a polynomial $p_{s,r}$ of degree $2r+3$ , with the properties;\\

$(i)$. $p_{s,r}$ is symmetric, that is $p_{s,r}(x)=p_{s,r}(-x)$, for $x\in\mathcal{R}$.\\

$(ii)$. $p_{s,r}(n)={1\over n^{s}}$, for $1\leq n\leq r$.\\

$(iii)$. $p_{s,r}^{(k)}(r)=h_{s,r}^{(k),+}(r)$, $(0\leq k\leq 2)$\\

$(iv)$. $p_{s,r}^{(k)}(-r)=h_{s,r}^{(k),-}(-r)$, $(0\leq k\leq 2)$\\

\end{lemma}

\begin{proof}

We let, for $1\leq j\leq 1+r$, $1\leq k\leq r$;\\

$\overline{A}_{r}=
\begin{pmatrix}
1&1&\ldots&1&\ldots&1&\\
1&2^{2}&\ldots&2^{2j}&\ldots&2^{2(r+1)}&\\
\ldots&\\
1&k^{2}&\ldots&k^{2j}&\ldots&k^{2(r+1)}&\\
\ldots&\\
1&r^{2}&\ldots&r^{2j}&\ldots&r^{2(r+1)}&\\
0&2r&\ldots&2jr^{2j-1}&\ldots&2(r+1)r^{2r+1}&\\
0&2&\ldots&2j(2j-1)r^{2j-2}&\ldots&2(r+1)(2r+1)r^{2r}&\\
\end{pmatrix}\\
$
\ \ \ \ \ \\

$\overline{b}_{s,r}=
\begin{pmatrix}
1^{-s}\\
2^{-s}\\
\ldots\\
k^{-s}\\
\ldots\\
r^{-s}\\
-sr^{-(1+s)}\\
s(s+1)r^{-(2+s)}\\
\end{pmatrix}\\
$

We have that $det(\overline{A}_{r})\neq 0$, hence, we can solve the equation $\overline{A}_{r}(\overline{a}_{s,r})=\overline{b}_{s,r}$. Let $p_{s,r}(x)=\sum_{j=0}^{r+1}(\overline{a}_{s,r})_{(j+1)}x^{2j}$. We have, by construction, that $p_{s,r}(-x)=p_{s,r}(x)$, and
 $p_{s,r}^{(k)}(r)=h_{s,r}^{(k),+}(r)$. As both $p_{s,r}$ and $h_{s,r}$  are symmetric, we also have that, $p_{s,r}^{(k)}(r)=h_{s,r}^{(k),-}(-r)$, as required.

 \end{proof}

 \begin{defn}
 \label{g}

 We define;\\

 $g_{s,r}(x)=h_{s,r}(x)$, (if $|x|\geq r$)\\

 $g_{s,r}(x)=p_{s,r}(x)$, (if $|x|\leq r$)\\

 \end{defn}

\begin{lemma}
\label{hypotheses}

We have that $g_{s,r}\in C^{2}(\mathcal{R})$, $g_{s,r}$ is symmetric, and, moreover, the hypotheses of Lemmas \ref{inversionconditions} and \ref{C2} hold for $g_{s,r}$.

\end{lemma}

\begin{proof}

The fact that $g_{s,r}\in C^{2}(\mathcal{R})$ follows immediately from Conditions $(iii)$ and $(iv)$ of Lemma \ref{polynomial}. The symmetry condition is a consequence of Condition $(i)$.  If $x\geq r$, we have that;\\

 $|g_{s,r}(x)|\leq |x^{-Re(s)}||x^{-Im(s)}|\leq |x|^{-4}$\\

Hence, as $g_{s,r}$ is symmetric, $|g_{s,r}(x)|\leq |x|^{-4}$, for $|x|\geq r$.\\

If $|x|\leq r$;\\

$|g_{s,r}(x)|=|p_{s,r}(x)|\leq r^{2r+2}\sum_{j=0}^{r+1}|(\overline{a}_{s,r})_{(j+1)}|\leq r^{2r+2}\sqrt{r+2}||\overline{a}_{s,r}||$\\

It follows that $sup_{x\in\mathcal{R}}(|x|^{4}|g_{s,r}(x)|)\leq max(1, r^{2r+6}\sqrt{r+2}||\overline{a}_{s,r}||)$. Similarly, as $g'_{s,r}(x)={-s\over x^{s+1}}$, $g''_{s}(x)={s(s+1)\over x^{s+2}}$, $|x|>r$, then, if $|x|>r$, we have that;\\

 $|g'_{s,r}(x)|\leq  |s||x|^{-5}$\\

 $|g''_{s,r}(x)|\leq  |s||s-1||x|^{-6}$\\

and, if $|x|\leq r$;\\

$|g'_{s,r}(x)|=|p'_{s,r}(x)|\leq r^{2r+2}(\sum_{j=1}^{r+1}|2j(\overline{a}_{s,r})_{(j+1)}|)\leq 2(r+1)r^{2r+2}\sqrt{r+2}||\overline{a}_{s,r}||$\\

$|g''_{s}(x)|=|p''_{s}(x)|\leq r^{2r+2}(\sum_{j=1}^{r+1}|2j(2j-1)(\overline{a}_{s,r})_{(j+1)}|)\leq (2r+2)(2r+1)r^{2r+2}\sqrt{r+2}||\overline{a}_{s}||$\\

so that;\\

$sup_{x\in\mathcal{R}}(|x|^{5}|g'_{s,r}(x)|)\leq max(|s|, 2(r+1)r^{2r+7}\sqrt{r+2}||\overline{a}_{s}||)$\\

$sup_{x\in\mathcal{R}}(|x|^{6}|g''_{s,r}(x)|)\leq max(|s||s-1|, (2r+2)(2r+1)r^{2r+8}\sqrt{r+2}||\overline{a}_{s}||)$ $(*)$\\

It follows that Lemmas \ref{inversionconditions} and \ref{C2} holds for $g_{s,r}$, with $C=D=max(|s||s-1|,(2r+2)(2r+1)r^{2r+8}\sqrt{r+2}||\overline{a}_{s}||$.\\

\end{proof}

\begin{defn}
\label{f}
We let $f_{s,r}(y)=\int_{\mathcal{R}}g_{s,r}(x)e^{2\pi ixy}dx$.\\

$R_{s,r,1}=\sum_{n\in\mathcal{Z}_{\neq 0}}(\int_{r}^{\infty}{e^{2\pi i nx}\over x^{s}}dx)$\\

$R_{s,r,2}=\sum_{n\in\mathcal{Z}_{\neq 0}}(\int_{r}^{\infty}{e^{-2\pi i nx}\over x^{s}}dx)$\\

$R_{s,r}=R_{s,r,1}+R_{s,r,2}$\\

$P_{s,r,1}=\sum_{n\in\mathcal{Z}_{\neq 0}}(\int_{0}^{r}p_{s,r}(x)e^{2\pi i nx}dx)$\\

$P_{s,r,2}=\sum_{n\in\mathcal{Z}_{\neq 0}}(\int_{0}^{r}p_{s,r}(x)e^{-2\pi i nx}dx)$\\

$P_{s,r}=P_{s,r,1}+P_{s,r,2}$\\

 \end{defn}

\begin{lemma}
\label{conclusions}
We have that $f_{s,r}$ is symmetric, and $f_{s,r}$ satisfies the conclusions of Lemmas \ref{converse} and \ref{C2}. Moreover;\\

$f_{s,r}(0)+P_{s,r}+R_{s,r}=p_{s,r}(0)+2\sum_{n=1}^{\infty}{1\over n^{s}}$\\

\end{lemma}

\begin{proof}

The second part follows immediately from Definition \ref{f}, and Lemmas \ref{hypotheses}, \ref{converse} and \ref{C2}. It follows that $f_{s,r}\in L^{1}(\mathcal{R})$, and;\\

$f_{s,r}(-y)=\int_{\mathcal{R}}g_{s,r}(x)e^{-2\pi ixy}dx$\\

$=\int_{\mathcal{R}}g_{s,r}(-x)e^{2\pi ixy}dx$\\

$=\int_{\mathcal{R}}g_{s,r}(x)e^{2\pi ixy}dx$\\

$=f_{s,r}(y)$\\

Hence, $f_{s,r}$ is symmetric. By Lemma \ref{sumformula} , we have that;\\

$\sum_{n\in\mathcal{Z}}f_{s,r}(n)=\sum_{n\in\mathcal{Z}}g_{s,r}(n)$\\

As both $f_{s,r}$ and $g_{s,r}$ are symmetric, using Definition \ref{g} and property $(ii)$ of Lemma \ref{polynomial}, we obtain;\\

$f_{s,r}(0)+P_{s,r}+R_{s,r}$\\

$=f_{s,r}(0)+\sum_{n\in\mathcal{Z}_{\neq 0}}f_{s,r}(n)$\\

$=g_{s,r}(0)+2(\sum_{n=1}^{\infty}g_{s,r}(n))$\\

$=p_{s,r}(0)+2(\sum_{n=1}^{\infty}{1\over n^{s}})$\\

\end{proof}

\begin{lemma}
\label{approx}
We have that;\\

$|R_{s,r}|\leq {2|s|^{2}\over 3(Re(s)+1)r^{Re(s)+1}}$\\

$P_{s,r}=p_{s,r}(0)+p_{s,r}(r)+2\sum_{l=1}^{r-1}p_{s,r}(l)-2\int_{0}^{r}p_{s,r}(x)dx$\\

\end{lemma}
\begin{proof}
We have that;\\

$R_{s,r,1}=\sum_{n\in\mathcal{Z}_{\neq 0}}{-1\over 2\pi i n r^{s}}+\sum_{n\in\mathcal{Z}_{\neq 0}}{s\over 2\pi i n}\int_{r}^{\infty}{e^{2\pi i nx}\over x^{s+1}}dx$\\

$=\sum_{n\in\mathcal{Z}_{\neq 0}}{-s\over r^{s+1}(2\pi i n)^{2}}+\sum_{n\in\mathcal{Z}_{\neq 0}}{s(s+1)\over (2\pi i n)^{2}}\int_{r}^{\infty}{e^{2\pi i nx}\over x^{s+2}}dx$\\

$={2s\over 4r^{s+1}\pi^{2}}\sum_{n=1}^{\infty}{1\over n^{2}}+D_{s,r,1}$\\

$={s\over 2r^{s+1}\pi^{2}}{\pi^{2}\over 6}+D_{s,r,1}$\\

$={s\over 12r^{s+1}}+D_{s,r,1}$\\

where;\\

$|D_{s,r,1}|\leq {|s(s+1)|C_{s,r,1}\over 4\pi^{2}}\sum_{n\in\mathcal{Z}_{\neq 0}}{1\over n^{2}}={|s(s+1)|C_{s,r,1}\over 12}$\\

and $C_{s,r,1}\leq \int_{r}^{\infty}{1\over |x^{s+2}|}dx$\\

$=\int_{r}^{\infty}{dx\over x^{Re(s)+2}}$\\

$={1\over (Re(s)+1)r^{Re(s)+1}}$\\

It follows that;\\

$|R_{s,r,1}|\leq {|s|\over 12r^{Re(s)+1}}+{|s(s+1)|\over 12 (Re(s)+1)r^{Re(s)+1}}$\\

$\leq {|s|(Re(s)+1)\over 12(Re(s)+1)r^{Re(s)+1}}+{|s(s+1)|\over 12(Re(s)+1)r^{Re(s)+1}}$\\

$={|s|(Re(s)+1)+|s(s+1)|\over 12(Re(s)+1)r^{Re(s)+1}}$\\

$\leq {2|s(s+1)|\over 12(Re(s)+1)r^{Re(s)+1}}$\\

$\leq {|s|^{2}\over 3(Re(s)+1)r^{Re(s)+1}}$\\

Similarly, $|R_{s,r,2}|\leq {|s|^{2}\over 3(Re(s)+1)r^{Re(s)+1}}$, so that $|R_{s,r}|\leq {2|s|^{2}\over 3(Re(s)+1)r^{Re(s)+1}}$.\\

We have that;\\

$P_{s,r,1}=\sum_{n\in\mathcal{Z}_{\neq 0}}(\int_{0}^{r}p_{s,r}(x)e^{2\pi i nx}dx$\\

$=\sum_{n\in\mathcal{Z}_{\neq 0}}(\sum_{l=0}^{r-1}\int_{l}^{l+1}p_{s,r}(x)e^{2\pi i nx}dx)$\\

$=\sum_{l=0}^{r-1}(\sum_{n\in\mathcal{Z}_{\neq 0}}(\int_{l}^{l+1}p_{s,r}(x)e^{2\pi i nx}dx))$\\

$=\sum_{l=0}^{r-1}(\sum_{n\in\mathcal{Z}_{\neq 0}}(\int_{0}^{1}p_{s,r}(x+l)e^{2\pi i n(x+l)}dx))$\\

$=\sum_{l=0}^{r-1}(\sum_{n\in\mathcal{Z}_{\neq 0}}(\int_{0}^{1}p_{s,r}(x+l)e^{2\pi i nx}dx))$\\

$=\sum_{l=0}^{r-1}(\sum_{n\in\mathcal{Z}_{\neq 0}}(p^{l}_{s,r})^{\vee}(n))$\\

$=\sum_{l=0}^{r-1}(\sum_{n\in\mathcal{Z}}(p^{l}_{s,r})^{\vee}(n)-\int_{0}^{1}p_{s,r}^{l}dx)$, (\footnote{If $f\in (C[0,1]\cap C^{2}(0,1))$, and there exist $\{a_{+,j},a_{-,j}:0\leq j\leq 2\}\subset\mathcal{C}$, with $lim_{x\rightarrow 0^{+}}f^{(j)}(x)=a_{+,j}$ and $lim_{x\rightarrow 1^{-}}f^{(j)}(x)=a_{-,j}$, $(\dag)$, for $0\leq j\leq 2$, then a classical result in the theory of Fourier series, says that;\\

$\lim_{N\rightarrow\infty}\sum_{n=-N}^{N}(f)^{(\wedge)}(n)e^{2\pi i nx}=f(x)$ ($x\in (0,1)$)\\

$\lim_{N\rightarrow\infty}\sum_{n=-N}^{N}(f)^{(\wedge)}(n)={a_{+,0}+a_{-,0}\over 2}$\\

We give a simple proof of this result. First observe that there exists a polynomial $p\in\mathcal{C}[x]$, with $deg(p)=5$, such that $p^{(j)}(0)=0$ and $p^{(j)}(1)=a_{-,j}-a_{+,j}$, for $0\leq j\leq 2$. This follows from the fact that we can find $\overline{c}\subset \mathcal{C}^{3}$, such that $\overline{M}\centerdot\overline{c}=\overline{a}$, where $\overline{a}(j)=a_{-,j-1}-a_{+,j-1}$, for $1\leq j\leq 3$, and;\\

$\overline{M}=
\begin{pmatrix}
1&1&1\\
3&4&5\\
6&12&20\\
\end{pmatrix}\\
$

as $det(\overline{M})\neq 0$, and, setting $p(x)=\sum_{k=0}^{2}c_{k}x^{3+k}$. We have that $p+f\in C^{2}(S^{1})$, in which case the result follows from \cite{dep2}. Hence, it is sufficient to verify the result for the powers $\{x^{k}:0\leq k\leq 5\}$. We have that, for $k\geq 1$, $n\in\mathcal{Z}_{\neq 0}$;\\

$\int_{0}^{1} x^{k}e^{-2\pi i nx}dx$\\

$[{x^{k}e^{-2\pi i nx}\over -2\pi i n}]_{0}^{1}+{k\over 2\pi i n}\int_{0}^{1} x^{k-1}e^{-2\pi i nx}dx$\\

${-1\over 2\pi i n}+{k\over 2\pi i n}\int_{0}^{1} x^{k-1}e^{-2\pi i nx}dx$\\

$\int_{0}^{1} x^{k}e^{-2\pi i nx}dx$\\

$=-(\sum_{l=1}^{k}{k!\over (k-l+1)!(2\pi i n)^{l}})+\int_{0}^{1}e^{-2\pi i nx}dx$\\

$=-(\sum_{l=1}^{k}{k!\over (k-l+1)!(2\pi i n)^{l}})$\\

$lim_{N\rightarrow\infty}\sum_{n=-N}^{N}(x^{k})^{\wedge}(n)$\\

$={1\over k+1}-2\sum_{l=1}^{k}\sum_{n=1}^{\infty}{k!\over (k-l+1)!(2\pi i n)^{l}}$\\

Case $k=1$, we obtain $S_{k}={1\over 2}$\\

$k=2$, $S_{k}={1\over 3}+{2.2\over 4\pi^{2}}(\sum_{n=1}^{\infty}{1\over n^{2}})$\\

$k=3$, $S_{k}={1\over 4}+{2.3\over 4\pi^{2}}(\sum_{n=1}^{\infty}{1\over n^{2}})$\\

$k=4$, $S_{k}={1\over 5}+{2.4\over 4\pi^{2}}(\sum_{n=1}^{\infty}{1\over n^{2}})-{2.24\over 16\pi^{4}}(\sum_{n=1}^{\infty}{1\over n^{4}})$\\

$k=5$, $S_{k}={1\over 6}+{2.5\over 4\pi^{2}}(\sum_{n=1}^{\infty}{1\over n^{2}})-{2.120\over 16\pi^{4}}(\sum_{n=1}^{\infty}{1\over n^{4}})$\\

Using Lemma \ref{cot},  we have that;\\

$\sum_{n=1}^{\infty}{1\over n^{2}}={-\pi [cot(\pi z)z]^{(2)}|_{0}\over 2.2!}={-\pi.-4\pi\over 6.2.2!}={\pi^{2}\over 6}$\\

$\sum_{n=1}^{\infty}{1\over n^{4}}={-\pi [cot(\pi z)z]^{(4)}|_{0}\over 2.4!}={-\pi.-48{\pi}^{3}\over 90.2.4!}={\pi^{4}\over 90}$\\

$S_{2}={1\over 3}+{2.2\over 4\pi^{2}}({\pi^{2}\over 6})={1\over 3}+{1\over 6}={1\over 2}$\\

$S_{3}={1\over 4}+{2.3\over 4\pi^{2}}{\pi^{2}\over 6}={1\over 4}+{1\over 4}={1\over 2}$\\

$S_{4}={1\over 5}+{2.4\over 4\pi^{2}}{\pi^{2}\over 6}-{2.24\over 16\pi^{4}}{\pi^{4}\over 90}={1\over 2}$\\

$S_{5}={1\over 6}+{2.5\over 4\pi^{2}}{\pi^{2}\over 6}-{2.60\over 16\pi^{4}}{\pi^{4}\over 90}={1\over 2}$\\})\\

$=\sum_{l=0}^{r-1}({p^{l}_{s,r}(1)+p^{l}_{s,r}(0)\over 2}-\int_{0}^{1}p_{s,r}^{l}dx)$\\

$=\sum_{l=0}^{r-1}({p_{s,r}(l+1)+p_{s,r}(l)\over 2}-\int_{0}^{1}p_{s,r}(x+l)dx)$\\

$=\sum_{l=0}^{r-1}({p_{s,r}(l+1)+p_{s,r}(l)\over 2}-\int_{l}^{l+1}p_{s,r}(x)dx)$\\

$={p_{s,r}(0)+p_{s,r}(r)\over 2}+\sum_{l=1}^{r-1}p_{s,r}(l)-\int_{0}^{r}p_{s,r}(x)dx$\\

Similarly;\\

$P_{s,r,2}={p_{s,r}(0)+p_{s,r}(r)\over 2}+\sum_{l=1}^{r-1}p_{s,r}(l)-\int_{0}^{r}p_{s,r}(x)dx$\\

so that;\\

$P_{s,r}=P_{s,r,1}+P_{s,r,2}$\\

$=p_{s,r}(0)+p_{s,r}(r)+2\sum_{l=1}^{r-1}p_{s,r}(l)-2\int_{0}^{r}p_{s,r}(x)dx$\\

\end{proof}

\begin{lemma}
\label{totalsum}
If $Re(s)\geq 4$, $r\geq 1$, we have that;\\

$\sum_{n=r}^{\infty}{1\over n^{s}}$\\

$=\int_{r}^{\infty}{dx\over x^{s}}+{p_{s,r}(r)\over 2}+{R_{s,r}\over 2}$\\

$={1\over (s-1)r^{s-1}}+{R_{s,r}\over 2}+{r^{s}\over 2}$\\

If $r\geq 2$;\\

$\sum_{n=1}^{r-1}{1\over n^{s}}$\\

$=\sum_{j=0}^{r+1}(\overline{a}_{s,r})_{j+1}({B_{2j+1}(r)\over 2j+1})$\\

\end{lemma}

\begin{proof}
The first claim is just a simple rearrangement of the claim in Lemma \ref{conclusions}, using Lemma \ref{approx}. We have that;\\

$\int_{r}^{\infty}{dx\over x^{s}}={1\over (s-1)r^{s-1}}$\\

and ${p_{s,r}(r)\over 2}={r^{s}\over 2}$, by property $(ii)$ in Lemma \ref{polynomial}.

Moreover;\\

$\sum_{n=1}^{r-1}{1\over n^{s}}$\\

$=\sum_{l=1}^{r-1}p_{s,r}(l)$\\

$=\sum_{l=0}^{r-1}\sum_{j=0}^{r+1}(\overline{a}_{s,r})_{j+1}l^{2j}$\\

$=\sum_{j=0}^{r+1}(\overline{a}_{s,r})_{j+1}(\sum_{l=0}^{r-1}l^{2j})$\\

$=\sum_{j=0}^{r+1}(\overline{a}_{s,r})_{j+1}({B_{2j+1}(r)-B_{2j+1}(0)\over 2j+1})$\\

$=\sum_{j=0}^{r+1}(\overline{a}_{s,r})_{j+1}({B_{2j+1}(r)\over 2j+1})$\\

\end{proof}

\begin{rmk}
\label{furtherideas}
Using Lemma \ref{approx}, we have that $lim_{r\rightarrow\infty}|R_{s,r}|=0$, hence, Lemma \ref{totalsum} reduces the calculation of $\sum_{n=1}^{\infty}{1\over n^{s}}$ to a calculation involving Bernoulli polynomials. Moreover, letting $A_{s,r}=\sum_{n=r}^{\infty}{1\over n^{s}}$, we have that;\\

$\int_{r}^{\infty}{dx\over |x^{s}|}\leq |A_{s,r}|\leq \int_{r-1}^{\infty}{dx\over |x^{s}|}$\\

$\int_{r}^{\infty}{dx\over x^{Re(s)}}\leq |A_{s,r}|\leq \int_{r-1}^{\infty}{dx\over |x^{s}|}$\\

${1\over (Re(s)-1)r^{Re(s)-1}}\leq |A_{s,r}|\leq {1\over 3(r-1)^{3}}$\\

Observing that;\\

 ${|s|^{2}\over 3(Re(s)+1)r^{Re(s)+1}}\leq {1\over (Re(s)-1)r^{Re(s)-1}}$\\

  if $r\geq |s|\sqrt{{(Re(s)-1)\over 3(Re(s)+1)}}$, we have that the estimate $\sum_{j=0}^{r+1}(\overline{a}_{s,r})_{j+1}({B_{2j+1}(r)\over 2j+1})+{1\over (s-1)r^{s-1}}+{1\over 2r^{s}}$ improves upon $\sum_{j=0}^{r+1}(\overline{a}_{s,r})_{j+1}({B_{2j+1}(r)\over 2j+1})$, for sufficiently large values of $r$. The coefficients $(\overline{a}_{s,r})_{j}$, $1\leq j\leq r+2$ can be computed using simple linear algebra. The computation of absolutely convergent Riemann sums, and their differences, occurs in the evaluation of $\zeta(s)$, for $0<Re(s)<1$, it is well known that $\zeta(s)\neq 0$, for $Re(s)\geq 1$. It is hoped that the above method might lead to some progress in the direction of solving the famous Riemann hypothesis, that, $\zeta(s)=0$ iff $Re(s)={1\over 2}$ or $s=-2w$, for $w\in\mathcal{Z}_{\geq 1}$, see \cite{A}.
\end{rmk}
\end{section}


\begin{thebibliography}{99}
\bibitem{A} Introduction to Analytical Number Theory, Springer, T. Apostol, (1976)\\ 
\bibitem{dep1} A Simple Proof of the Inversion Theorem using Nonstandard Analysis, available at http://www.magneticstrix.net, T. de Piro, (2013).\\
\bibitem{dep2} A Simple Proof of the Uniform Convergence of Fourier Series using Nonstandard Analysis, available at http://www.magneticstrix.net, T.de Piro, (2013).\\
\bibitem{SS} Fourier Analysis: An Introduction, Princeton Lecture Series, E. Stein, R. Shakarchi, (2003).\\



\end{thebibliography}
\end{document}